\documentclass[12pt,reqno]{amsart}
\usepackage{amsmath, amsfonts, amssymb, amsthm, hyperref}
\usepackage{bm}
\def\l{\left}
\def\r{\right}
\def\bg{\bigg}
\def\({\bg(}
\def\){\bg)}
\def\t{\text}
\def\f{\frac}

\def\ord{{\rm ord}}

\def\eq{\equiv}

\def\<{\langle}
\def\>{\rangle}
\def\1{{\bf 1}}

\theoremstyle{plain}
\newtheorem{theorem}{Theorem}
\newtheorem{conjecture}{Conjecture}

\newtheorem{lemma}{Lemma}

\theoremstyle{definition}

\theoremstyle{remark}
\newtheorem{Rem}{Remark}

\numberwithin{equation}{section}

\begin{document}
\title{Divisibility results concerning truncated hypergeometric series}
\author{Chen Wang}
\address[Chen Wang]{Department of Mathematics, Nanjing
University, Nanjing 210093, People's Republic of China}
\email{cwang@smail.nju.edu.cn}
\author{Wei Xia}
\address[Wei Xia]{Department of Mathematics, Nanjing
University, Nanjing 210093, People's Republic of China}
\email{mg1821010@smail.nju.edu.cn}

\begin{abstract}
In this paper, using the well-known Karlsson-Minton formula, we mainly establish two divisibility results concerning truncated hypergeometric series. Let $n>2$ and $q>0$ be integers with $2\mid n$ or $2\nmid q$. We show that
$$
\sum_{k=0}^{p-1}\frac{(q-\frac{p}{n})_k^n}{(1)_k^n}\equiv0\pmod{p^3}
$$
and
$$
p^n\sum_{k=0}^{p-1}\frac{(1)_k^n}{(\frac{p}{n}-q+2)_k^n}\equiv0\pmod{p^3}
$$
for any prime $p>\max\{n,(q-1)n+1\}$, where $(x)_k$ denotes the Pochhammer symbol defined by
$$
(x)_k=\begin{cases}1,\quad &k=0,\\
x(x+1)\cdots(x+k-1),\quad &k>0.\end{cases}
$$
Let $n\geq4$ be an even integer. Then for any prime $p$ with $p\equiv-1\pmod{n}$, the first congruence above implies that
$$
\sum_{k=0}^{p-1}\frac{(\frac{1}{n})_k^n}{(1)_k^n}\equiv0\pmod{p^3}.
$$
This confirms a recent conjecture of Guo.
\end{abstract}

\thanks{This work is supported by the National Natural Science Foundation of China (Grant No. 11971222).}
\thanks{The second author is the corresponding author.}

\maketitle

\section{Introduction}
\setcounter{lemma}{0} \setcounter{theorem}{0}
\setcounter{equation}{0}\setcounter{proposition}{0}
\setcounter{Rem}{0}\setcounter{conjecture}{0}

The truncated hypergeometric series are defined by
$$
{}_nF_{n-1}\bigg[\begin{matrix}a_1&a_2&\cdots&a_n\\ &b_1&\cdots&b_{n-1}\end{matrix}\bigg| z\bigg]_n=\sum_{k=0}^{n}\f{(a_1)_k(a_2)_k\cdots(a_n)_k}{(b_1)_k(b_2)_k\cdots(b_{n-1})_k}\f{z^k}{k!},
$$
where
$$
(x)_k=\begin{cases}1,\quad &k=0,\\
x(x+1)\cdots(x+k-1),\quad &k>0.\end{cases}
$$
denotes the so-called Pochhammer symbol (or rising factorial).
Clearly, they are truncations of the original hypergeometric series. In the past few decades, many interesting supercongruences concerning truncated hypergeometric series have been studied (for example, see \cite{DFLST16,Liu,LR,MaoS,MS,Su2,WP}).

In 2015, Sun \cite{Su2} studied some new congruences formally motivated by the well-known limit
$$
\lim_{n\rightarrow\infty}\l(1+\f{1}{n}\r)^n = e.
$$
For example, for any prime $p>3$, he showed that
\begin{equation}\label{suntheorem1}
\sum_{k=0}^{p-1}\f{(\f{1}{p+1})_k^{p+1}}{(k!)^{p+1}}\eq0\pmod{p^5};
\end{equation}
for any prime $p>3$ and integer $n$ with $p\nmid n$, he proved that
\begin{equation}\label{suntheorem2}
\sum_{k=0}^{p-1}\f{(1-\f{p}{n})_k^n}{(1)_k^n}\eq\f{(n-1)(7n-5)}{36n^2}p^4B_{p-3}\pmod{p^5},
\end{equation}
where $B_0,B_1,B_2,\ldots$ are the well-known Bernoulli numbers (cf. \cite{IR}).
Clearly, \eqref{suntheorem1} is just the special case of \eqref{suntheorem2}. Sun also studied a more general form of \eqref{suntheorem2} and proposed the following conjecture which was later confirmed by Meng and Sun \cite{MS}: for integers $n>2$ and $q>0$ with $n$ even or $q$ odd and primes $p>nq$ we have
\begin{equation}\label{sunconjecture}
\sum_{k=0}^{p-1}\f{(q-\f{p}{n})_k^n}{(1)_k^n}\eq0\pmod{p^3}.
\end{equation}

In 2019, Guo and Zudilin \cite{GuoZu} developed a unified method called $q$-microscope to deal with different $q$-supercongruences. Since then, by using the $q$-microscope, a series of challenging $q$-supercongruences has been established (see, for example, \cite{Guo19,Guo20a,GuoSch20,GuoZu}). In \cite{Guo19}, Guo obtained a $q$-analogue of the following congruence similar to \eqref{sunconjecture}: for any integer $d\geq2$ and $r\leq d-2$ such that $\gcd(r,d)=1$ and for any prime $p$ satisfying $p\eq-r\pmod{d}$ with $p\geq d-r$, we have
$$
\sum_{k=0}^{p-1}\f{(\f{r}{d})_k^d}{k!^d}\eq0\pmod{p^2}.
$$
It is clear that for some $r$ (for example, $r=-1$), \eqref{sunconjecture} implies that the above congruence holds modulo $p^3$. Noting this, Guo posed the following conjecture.
\begin{conjecture}\label{guoconj}
Let $d\geq4$ be an even integer. Then, for any prime $p$ with $p\eq-1\pmod{d}$,
\begin{equation}\label{guoconjecture}
\sum_{k=0}^{p-1}\f{(\f{1}{d})_k^d}{k!^d}\eq0\pmod{p^3}.
\end{equation}
\end{conjecture}

If we set $q=(p+1)/d$ in \eqref{sunconjecture}, then $dq>p$. Thus  \eqref{sunconjecture} can not imply \eqref{guoconjecture}. This is our first motivation. The following theorem confirms Conjecture \ref{guoconj} by establishing the generalization of \eqref{sunconjecture}.

\begin{theorem}\label{maintheorem1}
Let $n>2$ and $q>0$ be integers with $2\mid n$ or $2\nmid q$. Then for any prime $p>\max\{n,(q-1)n+1\}$ we have
\begin{equation}\label{th1eq1}
\sum_{k=0}^{p-1}\frac{(q-\frac{p}{n})_k^n}{(1)_k^n}\equiv0\pmod{p^3}.
\end{equation}
\end{theorem}
\begin{Rem}
\eqref{suntheorem1} ensures that \eqref{guoconjecture} holds for $d=p+1$. For $d<p+1$, \eqref{guoconjecture} follows immediately by letting $n=d$ and $q=(p+1)/d$ in Theorem \ref{maintheorem1}.
\end{Rem}

In 2016, Deines, Fuselier, Long, Swisher and Tu \cite{DFLST16} investigated some congruences for the truncated hypergeometric series. Especially, for any integer $n\geq3$ and prime $p\eq1\pmod{n}$, they conjectured that
\begin{equation}\label{DFLSTtheorem1}
\sum_{k=0}^{p-1}\f{(1-\f{1}{n})_k^n}{(1)_k^n}\eq-\Gamma_p\l(\f{1}{n}\r)^n\pmod{p^3}
\end{equation}
and
\begin{equation}\label{DFLSTtheorem2}
p^n\sum_{k=0}^{p-1}\f{(1)_k^n}{(1+\f{1}{n})_k^n}\eq-\Gamma_p\l(\f{1}{n}\r)^n\pmod{p^3},
\end{equation}
where $\Gamma_p(x)$ is the $p$-adic Gamma function introduced by Morita (see \cite{R} for details about $p$-adic Gamma function). These two conjectures have been proved by the first author and Pan \cite{WP} in 2018. In fact, \eqref{DFLSTtheorem1} is exactly the case $q=(p-1)/n+1$ in \eqref{th1eq1}. This indicates that the condition $p>(q-1)n+1$ in \eqref{guoconjecture} is the best when $p>n$.

Our second motivation comes from \eqref{DFLSTtheorem2}. We call \eqref{DFLSTtheorem2} the dual congruence of \eqref{DFLSTtheorem1}. The following theorem gives the dual congruence of \eqref{maintheorem1}.

\begin{theorem}\label{maintheorem2}
Under the same conditions of Theorem \ref{maintheorem1}, we also have
\begin{equation}\label{th2eq1}
p^n\sum_{k=0}^{p-1}\frac{(1)_k^n}{(\frac{p}{n}-q+2)_k^n}\equiv0\pmod{p^3}.
\end{equation}
\end{theorem}

Theorems \ref{maintheorem1} and \ref{maintheorem2} will be proved by using the Karlsson-Minton formula in Sections 2--3.

\section{Proof of Theorem \ref{maintheorem1}}
\setcounter{lemma}{0} \setcounter{theorem}{0}
\setcounter{equation}{0}\setcounter{proposition}{0}
\setcounter{Rem}{0}\setcounter{conjecture}{0}

The following identity due to Karlsson and Minton plays a key role in the subsequent proofs.

\begin{lemma}\cite[Eq. (1.9.2)]{GR}\label{karlsson_minton}
Let $m_1,m_2,\ldots,m_r$ be nonnegative integers and $a$ be any complex number such that $\Re(-a)>m_1+\cdots+m_r$. Then we have
$$
\sum_{k=0}^{\infty}\f{(a)_k(b_1+m_1)_k\cdots(b_r+m_r)_k}{(1)_k(b_1)_k\cdots(b_r)_k}=0.
$$
\end{lemma}

\begin{lemma}\label{harmonic}
Under the same conditions of Theorem \ref{maintheorem1}, we have
$$
\sum_{k=0}^{p-q}\f{(q)_k^n}{(1)_k^n}H_{q-1}^{(2)}\eq\sum_{k=0}^{p-q}\f{(q)_k^n}{(1)_k^n}H_{q+k-1}^{(2)}\eq\sum_{k=0}^{p-q}\f{(q)_k^n}{(1)_k^n}H_k^{(2)}\eq0\pmod{p},
$$
where $H_k^{(m)}=\sum_{j=1}^k1/j^m$ denotes the $k$th harmonic number of order $m$.
\end{lemma}
\begin{proof}
Set
$$
\Upsilon(x,y)=\sum_{k=0}^{p-1}\f{(1-p)_k(q+x)_k(q+y)_k(q)_k^{n-2}}{(1)_k^{n-1}(1+x)_k(1+y)_k}.
$$
Recall that $p>(q-1)n+1$. With the help of Lemma \ref{karlsson_minton} we have
$$
\sum_{k=0}^{p-q}\f{(q)_k^n}{(1)_k^n}\eq\Upsilon(0,0)=0\pmod{p}.
$$
Clearly, there also holds that
\begin{equation}\label{hq-1}
\sum_{k=0}^{p-q}\f{(q)_k^n}{(1)_k^n}H_{q-1}^{(2)}\eq0\pmod{p}.
\end{equation}
It is easy to check that
\begin{equation}\label{diff}
\f{d}{dx}(a+x)_k=(a+x)_k\sum_{i=0}^{k-1}\f{1}{a+i+x}.
\end{equation}
By \eqref{diff} one can directly verify that
\begin{gather}
\label{upsilonxy}\Upsilon''_{xy}(0,0)=\sum_{k=0}^{p-1}\f{(1-p)_k(q)_k^n}{(1)_k^{n+1}}\l(\sum_{i=0}^{k-1}\f{1}{q+i}-H_k^{(1)}\r)^2,\\
\label{upsilonxx}\Upsilon''_{xx}(0,0)=\sum_{k=0}^{p-1}\f{(1-p)_k(q)_k^n}{(1)_k^{n+1}}\l(\l(\sum_{i=0}^{k-1}\f{1}{q+i}-H_k^{(1)}\r)^2+H_k^{(2)}-\sum_{i=0}^{k-1}\f{1}{(q+i)^2}\r).
\end{gather}
Subtracting \eqref{upsilonxy} from \eqref{upsilonxx} and noting that $\Upsilon(x,y)=0$ we immediately arrive at
$$
\sum_{k=0}^{p-1}\f{(1-p)_k(q)_k^n}{(1)_k^{n+1}}\l(\sum_{i=0}^{k-1}\f{1}{(q+i)^2}-H_k^{(2)}\r)=0.
$$
Since $n>2$, we have
$$
\f{(1-p)_k(q)_k^n}{(1)_k^{n+1}}\l(\sum_{i=0}^{k-1}\f{1}{(q+i)^2}-H_k^{(2)}\r)\eq0\pmod{p}\quad\t{for any}\quad k\in\{p-q+1,\ldots,p-1\}.
$$
By \eqref{hq-1} we have
\begin{align}\label{hqkhk}
&\sum_{k=0}^{p-1}\f{(1-p)_k(q)_k^n}{(1)_k^{n+1}}\l(\sum_{i=0}^{k-1}\f{1}{(q+i)^2}-H_k^{(2)}\r)\notag\\
\eq&\sum_{k=0}^{p-q}\f{(q)_k^n}{(1)_k^{n}}(H_{q+k-1}^{(2)}-H_{q-1}^{(2)}-H_k^{(2)})\notag\\
\eq&\sum_{k=0}^{p-q}\f{(q)_k^n}{(1)_k^{n}}(H_{q+k-1}^{(2)}-H_k^{(2)})\eq0\pmod{p}.
\end{align}
On the other hand,
\begin{align}\label{hpqk}
&\sum_{k=0}^{p-q}\f{(q)_k^n}{(1)_k^{n}}H_k^{(2)}=\sum_{k=0}^{p-q}(-1)^{kn}\binom{-q}{k}^nH_k^{(2)}\notag\\
\eq&\sum_{k=0}^{p-q}(-1)^{kn}\binom{p-q}{k}^nH_k^{(2)}=\sum_{k=0}^{p-q}(-1)^{(p-q-k)n}\binom{p-q}{k}^nH_{p-q-k}^{(2)}\notag\\
\eq&\sum_{k=0}^{p-q}(-1)^{(p-q-k)n}\binom{-q}{k}^nH_{p-q-k}^{(2)}\notag\\
=&\sum_{k=0}^{p-q}\f{(q)_k^n}{(1)_k^{n}}H_{p-q-k}^{(2)}\pmod{p},
\end{align}
where the last step follows from the fact $2\mid n$ or $2\nmid q$.

By a classical result due to Wolstenholme \cite{W}, we know that $H_{p-1}^{(2)}\eq0\pmod{p}$ for any prime $p>3$. It follows that
$$
H_{p-q-k}^{(2)}=\sum_{j=1}^{p-q-k}\f{1}{j^2}=\sum_{j=q+k}^{p-1}\f{1}{(p-j)^2}\eq\sum_{j=q+k}^{p-1}\f{1}{j^2}\eq-H_{q+k-1}^{(2)}\pmod{p}
$$
for any prime $p>3$ and $q+k\leq p$. This together with \eqref{hpqk} gives that
$$
\sum_{k=0}^{p-q}\f{(q)_k^n}{(1)_k^{n}}H_k^{(2)}\eq-\sum_{k=0}^{p-q}\f{(q)_k^n}{(1)_k^{n}}H_{q+k-1}^{(2)}\pmod{p}.
$$
Combining this with \eqref{hqkhk} we immediately get
$$
\sum_{k=0}^{p-q}\f{(q)_k^n}{(1)_k^n}H_{q+k-1}^{(2)}\eq\sum_{k=0}^{p-q}\f{(q)_k^n}{(1)_k^n}H_k^{(2)}\eq0\pmod{p}.
$$
The proof of Lemma \ref{harmonic} is now complete.
\end{proof}

\medskip
\noindent{\it Proof of Theorem \ref{maintheorem1}}. Set
$$
\Phi(x,y)=\sum_{k=0}^{p-q}\f{(q-x)_k(q-y)_k^{n-1}}{(1)_k^n}\quad\t{and}\quad\Psi(x)=\sum_{k=0}^{p-q}\f{(q-x)_k^n}{(1)_k^n}.
$$
Clearly,
\begin{equation}\label{diff1}
\f{d}{dx}(q-x)_k=-(q-x)_k\sum_{i=0}^{k-1}\f{1}{q+i-x}.
\end{equation}
It follows that
\begin{equation}\label{dPsi}
\f{d\Psi(x)}{dx}=n\f{\partial\Phi(x,y)}{\partial x}\big|_{y=x}
\end{equation}
and
\begin{equation}\label{dPsi2}
\f{d^2\Psi(x)}{dx^2}=n\f{\partial^2\Phi(x,y)}{\partial x\partial y}\big|_{y=x}+n\f{\partial^2\Phi(x,y)}{\partial x^2}\big|_{y=x}.
\end{equation}
By Taylor expansion, we have
$$
\Psi\l(\f{p}{n}\r)=\Psi(0)+\f{\Psi'(0)}{n}p+\f{\Psi''(0)}{2n^2}p^2+\cdots+\f{\Psi^{(r)}(0)}{r!n^r}p^r+\cdots,
$$
where $\Psi',\Psi''$ and $\Psi^{(r)}$ stand for the first, the second and the $r$th derivatives of $\Psi(x)$ respectively. It is easy to see that $p\nmid (q+i)$ for all $i\in\{0,1,\ldots,p-q-1\}$. Thus it is not hard to find that $\ord_p(\Psi^{(r)}(0))\geq0$ for any nonnegative integer $r$. As we all know,
$$
\ord_p(r!)=\sum_{i=1}^{\infty}\l\lfloor\f{r}{p^i}\r\rfloor\leq\sum_{i=1}^{\infty}\f{r}{p^i}=\f{r}{p-1}.
$$
Furthermore, noting that $p\geq5$ we have
$$
\ord_p\l(\f{p^r}{r!}\r)\geq r-\f{r}{p-1}\geq\f{3}{4}r>2\quad\t{for any}\quad r\geq3.
$$
Since $n<p$, we know $\ord_p(n)=0$. The above discussion gives
$$
\ord_p\l(\f{\Psi^{(r)}(0)}{r!n^r}p^r\r)\geq3\quad\t{for any}\quad r\geq3.
$$
So we arrive at
\begin{equation}\label{PsiTaylor}
\Psi\l(\f{p}{n}\r)\eq\Psi(0)+\f{\Psi'(0)}{n}p+\f{\Psi''(0)}{2n^2}p^2\pmod{p^3}.
\end{equation}
Via a similar discussion as above, we can also obtain
\begin{equation}\label{PhiTaylor}
\Phi(p,0)\eq\Phi(0,0)+\Phi'_{x}(0,0)p+\f{\Phi''_{xx}(0,0)}{2}p^2\pmod{p^3}.
\end{equation}
Now combining \eqref{dPsi}--\eqref{PhiTaylor} we have
\begin{align}\label{key1}
\sum_{k=0}^{p-1}\f{(q-\f{p}{n})_k^n}{(1)_k^n}\eq&\Psi\l(\f{p}{n}\r)\notag\\
\eq&\Psi(0)+\f{\Psi'(0)}{n}p+\f{\Psi''(0)}{2n^2}p^2\notag\\
=&\Phi(0,0)+\Phi'_x(0,0)p+\f{\Phi''_{xy}(0,0)+\Phi''_{xx}(0,0)}{2n}p^2\notag\\
\eq&\Phi(p,0)+\f{\Phi''_{xy}(0,0)-(n-1)\Phi''_{xx}(0,0)}{2n}p^2\pmod{p^3}.
\end{align}
Note that $(q-1)(n-1)=(q-1)n+1-q<p-q$. Thus by Lemma \ref{karlsson_minton} we have
\begin{equation}\label{phip0}
\Phi(p,0)=\sum_{k=0}^{p-q}\f{(q-p)_k(q)_k^{n-1}}{(1)_k^n}=\sum_{k=0}^{\infty}\f{(q-p)_k(q)_k^{n-1}}{(1)_k^n}=0.
\end{equation}
In view of \eqref{diff1} we also obtain that
\begin{gather}
\label{phixy}\Phi''_{xy}(0,0)=(n-1)\sum_{k=0}^{p-q}\f{(q)_k^n}{(1)_k^n}\l(\sum_{i=0}^{k-1}\f{1}{q+i}\r)^2,\\
\label{phixx}\Phi''_{xx}(0,0)=\sum_{k=0}^{p-q}\f{(q)_k^n}{(1)_k^n}\l(\l(\sum_{i=0}^{k-1}\f{1}{q+i}\r)^2-\sum_{i=0}^{k-1}\f{1}{(q+i)^2}\r).
\end{gather}
Substituting \eqref{phip0}--\eqref{phixx} into \eqref{key1} we arrive at
\begin{equation*}
\sum_{k=0}^{p-1}\f{(q-\f{p}{n})_k^n}{(1)_k^n}\eq\f{n-1}{2n}p^2\sum_{k=0}^{p-q}\f{(q)_k^n}{(1)_k^n}\sum_{i=0}^{k-1}\f{1}{(q+i)^2}\pmod{p^3}.
\end{equation*}
Finally, Theorem \ref{maintheorem1} follows from Lemma \ref{harmonic}.\qed
\medskip
\section{Proof of Theorem \ref{maintheorem2}}
\setcounter{lemma}{0} \setcounter{theorem}{0}
\setcounter{equation}{0}\setcounter{proposition}{0}
\setcounter{Rem}{0}\setcounter{conjecture}{0}

We need the following lemmas.

\begin{lemma}\label{harmonic2}
Under the same conditions of Theorem \ref{maintheorem1}, we have
\begin{gather}
\label{harmonic2-1}\sum_{k=0}^{p-q}\f{(q)_k^n}{(1)_k^n}(H_k^{(1)}-H_{q+k-1}^{(1)})\eq0\pmod{p},\\
\label{harmonic2-2}\sum_{k=0}^{p-q}\f{(q)_k^n}{(1)_k^n}\l((H_k^{(1)})^2-(H_{q+k-1}^{(1)})^2\r)\eq0\pmod{p}.
\end{gather}
\end{lemma}
\begin{proof}
Clearly,
\begin{align*}
&\sum_{k=0}^{p-q}\f{(q)_k^n}{(1)_k^{n}}H_k^{(1)}=\sum_{k=0}^{p-q}(-1)^{kn}\binom{-q}{k}^nH_k^{(1)}\\
\eq&\sum_{k=0}^{p-q}(-1)^{kn}\binom{p-q}{k}^nH_k^{(1)}=\sum_{k=0}^{p-q}(-1)^{(p-q-k)n}\binom{p-q}{k}^nH_{p-q-k}^{(1)}\\
\eq&\sum_{k=0}^{p-q}(-1)^{(p-q-k)n}\binom{-q}{k}^nH_{p-q-k}^{(1)}\\
=&\sum_{k=0}^{p-q}\f{(q)_k^n}{(1)_k^{n}}H_{p-q-k}^{(1)}\pmod{p}.
\end{align*}
It is well-known (cf. \cite{W}) that $H_{p-1}^{(1)}\eq0\pmod{p^2}$ for any prime $p>3$. Thus we have
$$
H_{p-q-k}^{(1)}=\sum_{j=1}^{p-q-k}\f{1}{j}=\sum_{j=q+k}^{p-1}\f{1}{p-j}\eq-\sum_{j=q+k}^{p-1}\f{1}{j}\eq H_{q+k-1}^{(1)}\pmod{p}
$$
provided that $q+k\leq p$. In view of the above, \eqref{harmonic2-1} holds. Here we shall not give the proof of \eqref{harmonic2-2} since it can be verified in a similar way.
\end{proof}

\begin{lemma}\label{harmonic3}
Under the same conditions as the ones in Theorem \ref{maintheorem1}, we have
\begin{gather}
\label{harmonic3-1}\sum_{k=0}^{p-q}\f{(q-\f{p}{n})_k^n}{(1)_k^n}\l(\sum_{i=0}^{k-1}\f{1}{q+i-\f{p}{n}}-H_k^{(1)}\r)\eq0\pmod{p^2},\\
\label{harmonic3-2}\sum_{k=0}^{p-q}\f{(q-\f{p}{n})_k^n}{(1)_k^n}\l(\sum_{i=0}^{k-1}\f{1}{q+i-\f{p}{n}}-H_k^{(1)}\r)^2\eq0\pmod{p}.
\end{gather}
\end{lemma}

\begin{proof}
Let $\Upsilon(x,y)$ be defined as in the proof of Lemma \ref{harmonic}. \eqref{harmonic3-2} follows from \eqref{upsilonxy} and the fact $\Upsilon(x,y)=0$ immediately. Below we consider \eqref{harmonic3-1}. It is clear that
\begin{align}\label{upsilonx}
0=&\Upsilon'_x(0,0)=\sum_{k=0}^{p-1}\f{(1-p)_k(q)_k^n}{(1)_k^{n+1}}\l(\sum_{i=0}^{k-1}\f{1}{q+i}-H_k^{(1)}\r)\notag\\
\eq&\sum_{k=0}^{p-q}\f{(q)_k^n}{(1)_k^n}(1-pH_k^{(1)})\l(\sum_{i=0}^{k-1}\f{1}{q+i}-H_k^{(1)}\r)\notag\\
=&\sum_{k=0}^{p-q}\f{(q)_k^n}{(1)_k^n}\l(\sum_{i=0}^{k-1}\f{1}{q+i}-H_k^{(1)}-pH_k^{(1)}\sum_{i=0}^{k-1}\f{1}{q+i}+p(H_k^{(1)})^2\r)\pmod{p^2}.
\end{align}
Now by Lemmas \ref{harmonic}, \ref{harmonic2} and \eqref{upsilonx} we obtain
\begin{align}\label{key4}
&\sum_{k=0}^{p-q}\f{(q-\f{p}{n})_k^n}{(1)_k^n}\l(\sum_{i=0}^{k-1}\f{1}{q+i-\f{p}{n}}-H_k^{(1)}\r)\notag\\
\eq&\sum_{k=0}^{p-q}\f{(q)_k^n}{(1)_k^n}\l(1-p\sum_{i=0}^{k-1}\f{1}{q+i}\r)\l(\sum_{i=0}^{k-1}\f{1}{q+i}+\f{p}{n}\sum_{i=0}^{k-1}\f{1}{(q+i)^2}-H_k^{(1)}\r)\notag\\
\eq&\sum_{k=0}^{p-q}\f{(q)_k^n}{(1)_k^n}\l(\sum_{i=0}^{k-1}\f{1}{q+i}-H_k^{(1)}+pH_k^{(1)}\sum_{i=0}^{k-1}\f{1}{q+i}-p\l(\sum_{i=0}^{k-1}\f{1}{q+i}\r)^2\r)\notag\\
\eq&2\sum_{k=0}^{p-q}\f{(q)_k^n}{(1)_k^n}\l(pH_k^{(1)}\sum_{i=0}^{k-1}\f{1}{q+i}-p(H_k^{(1)})^2\r)\pmod{p^2}.
\end{align}
By Theorem \ref{maintheorem1}, Lemma \ref{harmonic2} and  \eqref{upsilonxy} we have
\begin{align}\label{key5}
0\eq&\sum_{k=0}^{p-1}\f{(q)_k^n}{(1)_k^n}\l(\sum_{i=0}^{k-1}\f{1}{q+i}-H_k^{(1)}\r)^2\eq\sum_{k=0}^{p-q}\f{(q)_k^n}{(1)_k^n}\l(H_{q+k-1}-H_{q-1}-H_k^{(1)}\r)^2\notag\\
\eq&2\sum_{k=0}^{p-q}\f{(q)_k^n}{(1)_k^n}\l((H_k^{(1)})^2-H_k^{(1)}\sum_{i=0}^{k-1}\f{1}{q+i}\r)\pmod{p}.
\end{align}
Substituting \eqref{key5} into \eqref{key4}, \eqref{harmonic3-1} follows.

The proof of Lemma \ref{harmonic3} is now complete.
\end{proof}

\medskip
\noindent{\it Proof of Theorem \ref{maintheorem2}}.
It is easy to check that for any $k=0,1,\ldots,p-1$,
$$
\f{(1)_k}{(\f{p}{n}-q+2)_k}=\f{(1)_{p-1}}{(\f{p}{n}-q+2)_{p-1}}\f{(q-\f{p}{n}-p)_{p-1-k}}{(1-p)_{p-1-k}}.
$$
Thus
$$
p^n\sum_{k=0}^{p-1}\f{(1)_k^n}{(\f{p}{n}-q+2)_k^n}=\f{p^n(1)_{p-1}^n}{(\f{p}{n}-q+2)_{p-1}^n}\sum_{k=0}^{p-1}\f{(q-\f{p}{n}-p)_k^n}{(1-p)_k^n}.
$$

We first illustrate that the proof of the case $q=1$ is trivial. If $q=1$, then $p\nmid(p/n-q+2)_{p-1}^n$. Since $n\geq3$, we immediately obtain that
$$
p^n\sum_{k=0}^{p-1}\f{(1)_k^n}{(\f{p}{n}-q+2)_k^n}\eq0\pmod{p^3}.
$$

Below we suppose that $q>1$. Now we have
$$
\f{p^n(1)_{p-1}^n}{(\f{p}{n}-q+2)_{p-1}^n}\not\eq0\pmod{p}.
$$
Thus it suffices to show
\begin{equation}\label{key2}
\sum_{k=0}^{p-1}\f{(q-\f{p}{n}-p)_k^n}{(1-p)_k^n}\eq0\pmod{p^3}.
\end{equation}
Set
$$
\Delta(x)=\sum_{k=0}^{p-q}\f{(q-\f{p}{n}+x)_k^n}{(1+x)_k^n}.
$$
Via a similar argument as the one in the proof of Theorem \ref{maintheorem1}, we have
\begin{align*}
\sum_{k=0}^{p-1}\f{(q-\f{p}{n}-p)_k^n}{(1-p)_k^n}\eq&\Delta(-p)\eq\Delta(0)-\Delta'(0)p+\f{\Delta''(0)}{2}p^2\pmod{p^3},
\end{align*}
where
\begin{gather*}
\Delta'(0)=n\sum_{k=0}^{p-q}\f{(q-\f{p}{n})_k^n}{(1)_k^n}\l(\sum_{i=0}^{k-1}\f{1}{q+i-\f{p}{n}}-H_k^{(1)}\r),\\
\Delta''(0)=\sum_{k=0}^{p-q}\f{(q-\f{p}{n})_k^n}{(1)_k^n}\l(n^2\l(\sum_{i=0}^{k-1}\f{1}{q+i-\f{p}{n}}-H_k^{(1)}\r)^2+n\l(H_k^{(2)}-\sum_{i=0}^{k-1}\f{1}{(q+i-\f{p}{n})^2}\r)\r).
\end{gather*}
By Theorem \ref{maintheorem1} we have $\Delta(0)\eq0\pmod{p^3}$. Then \eqref{key2} follows from Lemmas \ref{harmonic} and \ref{harmonic3}.

The proof of Theorem \ref{maintheorem2} is now complete.\qed

\end{document}